\documentclass[reqno, 11pt]{amsart}
\usepackage{url}
\usepackage[margin=1in]{geometry}
\usepackage[hidelinks]{hyperref}
\usepackage{cleveref}
 \usepackage{relsize}
 \usepackage[dvipsnames]{xcolor}
\usepackage{tikz}
\usetikzlibrary{arrows.meta, positioning, decorations.markings}
\usepackage{tikz-cd}
\usepackage{pgfplots}
\usepgfplotslibrary{groupplots}
\pgfplotsset{compat=1.18}

\usetikzlibrary{3d, decorations.pathmorphing, fadings, backgrounds,decorations.pathreplacing}
\usepackage{pgfmath}
\usepackage{tikz-3dplot}
\usepackage{enumitem}
\usepackage{helvet}
\usepackage{color,soul}

\usepackage[style=numeric,maxbibnames=99]{biblatex}
\usepackage{caption} 
\usepackage{float}
\newtheorem{mainthm}{Theorem}

\Crefname{mainthm}{Theorem}{Mainthms}

\Crefname{maincor}{Corollary}{Maincors}

\Crefname{mainex}{Example}{Mainexs}
\newtheorem{theorem}{Theorem}[section]
\newtheorem{lemma}[theorem]{Lemma}
\newtheorem{prop}[theorem]{Proposition}
\Crefname{prop}{Proposition}{Propsitions}

\newtheorem{cor}[theorem]{Corollary}
\Crefname{cor}{Corollary}{corollarys}

\newtheorem{remark}[theorem]{Remark}

\theoremstyle{definition}
\newtheorem{definition}[theorem]{Definition}

\newtheoremstyle{break}
  {\topsep}{\topsep}%
  {\itshape}{}%
  {\bfseries}{}%
  {\newline}{}%
\theoremstyle{break}

\numberwithin{equation}{section}

\usepackage{bbm}

\usepackage{euscript}

\usepackage{pb-diagram}

\usepackage{amsmath}

\usepackage{amsxtra}
\usepackage{amssymb}
\usepackage{pifont}
\usepackage{amsbsy}

\usepackage{graphicx}
\usepackage{epstopdf}
\usepackage{colortbl}
\usepackage{multirow}
\usepackage{hhline}

\newskip\aline \newskip\halfaline
\newcommand{\bN}{{{\mathbb N}}}

\newcommand\bR{{\mathbb R}}

\newcommand\bS{{\mathbb S}}

\newcommand\bZ{{\mathbb Z}}

\DeclareMathOperator{\BiRic}{BiRic}

\newcommand{\eps}{{\epsilon}}

\def\inpr{\mathbin{\hbox to 6pt{\vrule height0.4pt width5pt depth0pt \kern-.4pt \vrule height6pt width0.4pt depth0pt\hss}}}

\newcommand{\Ric}{\mathrm{Ric}}
\newcommand{\cA}{\mathcal{A}}
\newcommand{\Lip}{\mathrm{Lip}}

\def\XXint#1#2#3{{\setbox0=\hbox{$#1{#2#3}{\int}$ }
\vcenter{\hbox{$#2#3$ }}\kern-.59\wd0}}

\title[Topology of tame contractible $n$-manifolds]{Topology of tame contractible $n$-manifolds, $n=6,7$, with Positive BiRicci Curvature and Ricci lower bounds}
\author{Paul Sweeney Jr.}
\address[Paul Sweeney Jr.]
{ Department of Mathematics, Michigan State University, 619 Red Cedar Road, C212 Wells Hall, East Lansing, MI 48824, USA}
\email{sween245@msu.edu}

\begin{document}
\begin{abstract}
    We show that if an open smooth $6$-manifold $M^6$ is the interior of a compact, contractible manifold with boundary $X^6$, such that $\pi_3(X,\partial X)=0$, and $M$ supports a complete metric with uniform positive BiRicci curvature and Ricci curvature bounded below, then $M$ is diffeomorphic to $\bR^6$. Additionally, we show that if an open smooth $n$-manifold $M^n$, $n\in\{6,7\}$, is the interior of a compact, contractible manifold with boundary $X^n$, such that $\pi_3(X,\partial X)=0$, and $M$ supports a complete metric with positive BiRicci curvature with at most $C$-quadratic decay at infinity for some $C>C_n$ and nonnegative Ricci curvature, then $M$ is diffeomorphic to $\bR^n$. Lastly, for each $n\geq6$, we construct a compact, contractible Newman $n$-manifold, $N^n$, based on a presentation of the Higman group such that $\pi_i(N,\partial N)=0$, $3\leq i \leq n-3$, and $\mathrm{int}(N)$ admits a complete metric with uniform positive scalar curvature and Ricci curvature bounded below.
\end{abstract}
\maketitle
\section{Introduction}

An emblematic theme of Riemannian geometry is that the topology of a smooth manifold is constrained when it admits a Riemannian metric satisfying a positive curvature condition. A well-known example of this phenomenon is the result of Cohn-Vossen \cite{CohnVossen} and Huber \cite{HuberOnsubharmonic}, which shows that the only open\footnote{An open manifold is a non-compact manifold without boundary.} 2-manifold supporting a {complete} Riemannian metric with positive curvature is $\bR^2$. Likewise in higher dimensions, we would like to investigate what curvature conditions determine $\bR^n$. There have been a litany of results, recently, about the topological rigidity of $\bR^n$ under the condition of positive scalar curvature (see \Cref{subsec: openandposcurv} below).

Our goal is to study tame\footnote{A manifold is called tame if it is homeomorphic to the interior of a compact manifold with boundary.}  {contractible}\footnote{A smooth $n$-manifold $M^n$ is contractible if it is homotopy equivalent to a point, equivalently if all the homotopy groups of $M$ are trivial.} manifolds and show a topological rigidity theorem for $\bR^n$, $n=6,7$ under positive BiRicci curvature and Ricci curvature lower bounds. In \cite{ShenYeOnStableMinimalBiRicci}, Shen--Ye introduced the concept of BiRicci curvature and studied its relationship with minimal hypersurfaces. For $n\in\{3,4,5\},$ Shen--Ye showed that closed, connected, two-sided stable minimal hypersurfaces of Riemannian $n$-manifolds with $\BiRic\geq n-2$ have a uniform diameter estimate. They further studied BiRicci curvature in \cite{ShenYeOnGeometryandTopologyBiRicci}.

\begin{definition}
   Given a Riemannian manifold $(M^n,g)$, $p\in M$, and orthonormal $e_1,e_2\in T_pM$, the BiRicci curvature at $p$ of $\{e_1,e_2\}$ is defined as 
   \[
   \BiRic_g(e_1,e_2):=\Ric_g(e_1,e_1)+\Ric_g(e_2,e_2)-\sec_g(e_1,e_2)
   \]
   and for $p\in M$, we define $\BiRic_g(p):=\inf\left\{\BiRic_g(e_i,e_j)|\{e_i,e_j\} \text{ are orthonormal in } T_pM\right\}.$
   
\end{definition}
\begin{remark}
    Let $(M^n,g)$  be a Riemannian $n$-manifold such that $\BiRic\geq \kappa$. Then $R_g\geq\frac{n(n-1)}{2n-3}\kappa.$ Moreover, when $n\in\{2,3\}$, $\BiRic_g(p)=\frac{1}{2}R_g(p)$.
\end{remark}

We now introduce the concept of positive BiRicci curvature with at most quadratic decay at infinity, which is a strengthening of positive BiRicci curvature on an open manifold. This definition is related to the definition of positive scalar curvature with at most quadratic decay as defined in \cite[Definition 1.2]{BalacheffMorenodeMoraSardàSabourauComplete}, \cite[Definition 1.3.5]{MorenodeMoraSardàThesis}, and \cite[Definition 1.1]{ChenOptimal}.

\begin{definition} \label{def: BiRic decay}
    Let $(M^n,g)$ be a complete Riemannian $n$-manifold. Fix a base point $x_0\in M$, and denote by $r_{x_0}(y)=d(x_0,y)$ the distance function to $x_0$. The BiRicci curvature of $M$ has \emph{at most $C$-quadratic decay at infinity} if there exists a constant $D_0>0$ such that for every $y\in M$ with $r_{x_0}(y)\geq D_0$,
    \[
    \BiRic(y)\geq \frac{C}{r_{x_0}(y)^2}.
    \]
\end{definition}

In \cite{ZhouZhu}, Zhou and Zhu study complete manifolds with uniform BiRicci curvature and nonnegative Ricci curvature. In particular, they show the optimal volume growth for complete Riemannian $n$-manifolds, $3\leq n\leq 7$, with nonnegative Ricci curvature and BiRicci curvature bounded from below by $n-2$ outside a compact set. In this paper, we are able to show a topological rigidity result for $\bR^n$, $n=6,7$, under similar curvature control. In particular, our main results are the following.

\begin{mainthm}\label{t: main}
    Let $M^6$ be the interior of a compact, contractible $6$-manifold with boundary $X^6$ such that $\pi_3(X,\partial X)=0$. If M supports a complete metric with $\BiRic\geq \kappa>0$ and $\Ric\geq -K$ for some $K\geq0$, then $M$ is diffeomorphic to $\bR^6$.
\end{mainthm}

\begin{mainthm}\label{t: main2}
    Let $M^n$, $n\in\{6,7\}$, be the interior of a compact, contractible $n$-manifold with boundary $X^n$ such that $\pi_3(X,\partial X)=0$. If M supports a complete metric with positive BiRicci curvature with at most $C$-quadratic decay at infinity for some $C>C_n$, where $C_6=\frac{32}{21}$ and $C_7=\frac{25}{8}$ and $\Ric\geq0$, then $M$ is diffeomorphic to $\bR^n$.
\end{mainthm}

We note that Navarro, Pan, and Zhu study complete manifolds with nonnegative Ricci curvature and give restrictions on the fundamental group under the additional conditions of linear volume growth or sublinear diameter growth \cite{NavarroPanZhu1, NavarroPanZhu2, NavarroPanZhu3}. (See also the work of Cucinotta--Magnabosco--Semola \cite{CucinottaMagnaboscoSemola}, Antonelli \cite{antonelli2026universalvolumegrowthbounds}, Ge \cite{ge2026heatkernelgeometrygromovs}, Koirala \cite{koirala2026volumegrowthpositiveintermediate}, and Kong--Zhu \cite{kong2026positivescalarcurvaturevolume}.) Additionally, Deng \cite[Theorem E]{DengConformallyFlat} recently showed that a sufficiently connected, complete, open, locally conformally flat, Riemannian $n$-manifold with nonnegative scalar curvature is $\bR^n$.

Furthermore, we are able to show the existence of the manifolds that satisfy the topological hypotheses of \Cref{t: main} and admit metrics of uniform positive scalar curvature and Ricci curvature bounded below. These manifolds are the so-called Newman manifolds \cite{NewmanBoundaries} constructed based on a presentation of the Higman group \cite{HigmanAFinitelyGenerated}. We note that recently Deng \cite[Theorem F]{DengConformallyFlat} constructed complete, open, contractible, locally
conformally flat $n$-manifolds, $n\geq 6$, with uniform positive scalar curvature that are not homeomorphic to $\bR^n$. Deng's construction is a Newman manifold based on a presentation of the binary icosahedral group.

\begin{mainthm}\label{t: example}
    For $n\geq 6$, there exists a compact, contractible $n$-manifold with boundary $N^n$ such that $\pi_2(N,\partial N)\neq 0$ and $\pi_i(N,\partial N)=0$, $3\leq i\leq n-3$. In particular, $M=\mathrm{int}(N)$ is not diffeomorphic to $\bR^n$. Moreover, $M$ admits a complete Riemannian metric with uniform positive scalar curvature and Ricci curvature bounded below.
\end{mainthm}

\subsection{Closed Manifolds and Positive Curvature}  We recall some relevant results about closed manifolds under positive curvature conditions, where the landscape is considerably more developed. In dimension two, the Gauss--Bonnet theorem implies that a closed surface supporting a Riemannian metric with positive (scalar) curvature must be diffeomorphic to either the 2-sphere or the real projective plane. 

In dimension three, the foundational work by Schoen--Yau \cite{SchoenYauOntheStructure}, Gromov--Lawson \cite{GromovLawsontheclassification}, and Perelman \cite{PerelmanRicciFlowwithSurgery, PerelamanFiniteExtinction, PerelmanEntropy} establishes that a closed, orientable 3-manifold supports a Riemannian metric with positive scalar curvature if and only if it is diffeomorphic to a connected sum of the form
\[
\bS^3\#\left(\#_{j=1}^J \bS^3/\Gamma_j\right) \# \left(\#_{k=1}^K \bS^2 \times \bS^1\right),
\]
where each $\Gamma_j$ is a nontrivial finite subgroup of $\mathrm{SO}(4)$ acting freely on $\bS^3$.

In higher dimensions, there remain topological restrictions imposed by positive scalar curvature. Chodosh, Li, and Liokumovich \cite{ChodoshLiLiokumovich45} showed that if a closed 4-manifold $M$ satisfies $\pi_2(M) = 0$ and supports a Riemannian metric with positive scalar curvature, then there is a finite cover $\hat{M}$ of $M$ that is homotopy equivalent to either $\bS^4$ or a connected sum of finitely many copies of $\bS^3 \times \bS^1$. Similarly, in dimension five, Chodosh, Li, and Liokumovich \cite{ChodoshLiLiokumovich45} proved that if a closed manifold $M$ satisfies $\pi_2(M) =0$, $\pi_3(M) = 0$, and supports a Riemannian metric with positive scalar curvature, then there is a finite cover of $M$ that is homotopy equivalent to either $\bS^5$ or a connected sum of finitely many copies of $\bS^4 \times \bS^1$. (See also the work of Chen--Chu--Zhu \cite{ChenChuZhuPositivescalarcurvaturemetrics}.)

In \cite{MazurowskiWangYaoTopIntermediateCurv}, Mazurowski--Wang--Yao study closed manifolds which admit Riemannian metrics with positive $m$-intermediate curvature. They show, \cite[Theorem 7]{MazurowskiWangYaoTopIntermediateCurv}, that if $M^5$ is a closed manifold such that $\pi_2(M)=0$ and admits a metric with positive BiRicci curvature, then there is a finite cover $\hat{M}$ of $M$ that is homeomorphic to either $\bS^5$ or a connected sum of finitely many copies of $\bS^4 \times \bS^1$. They also showed that if $M^6$ is a closed manifold such that $\pi_2(M)=\pi_3(M)=\pi_4(M)=0$ and admits a metric with positive $4$-intermediate curvature, then there is a finite cover $\hat{M}$ of $M$ that is homeomorphic to either $\bS^6$ or a connected sum of finitely many copies of $\bS^5 \times \bS^1$.
% \begin{remark}
%     We note that \Cref{t: main} can be seen as an analogous result to \cite[Theorem 1]{ChodoshLiLiokumovich45} and \cite[Theorem 7]{MazurowskiWangYaoTopIntermediateCurv} for open $n$-manifolds, $n=6,7$.
% \end{remark}

\subsection{Open Manifolds and Positive Curvature}\label{subsec: openandposcurv}
To further contextualize the results, we recall some results about open manifolds and positive curvature. As mentioned above, Cohn-Vossen \cite{CohnVossen} and Huber \cite{HuberOnsubharmonic} showed that the only open 2-manifold supporting a {complete} Riemannian metric with positive curvature is $\bR^2$.

In dimension three, Gromov and Lawson showed that an open, orientable $3$-manifold with a finitely generated fundamental group supporting a complete metric with uniform positive scalar curvature is simply-connected at infinity \cite[Corollary 10.9]{GromovLawsonPositiveScalar}. In particular, this combined with the resolution of the Poincar\'e conjecture by Perelman in dimension three \cite{PerelmanRicciFlowwithSurgery, PerelamanFiniteExtinction, PerelmanEntropy} and a result of Edwards \cite{Edwards}, implies that an open, {contractible} 3-manifold with a complete Riemannian metric of uniform positive scalar curvature must be diffeomorphic to $\bR^3$. Chang, Weinberger, and Yu \cite{ChangWeinbergerYuTaming} were able to prove the same result using index theory. An application of Deng's recent result \cite[Theorem E]{DengConformallyFlat} implies that every complete, contractible, locally conformally flat 3-manifold with nonnegative scalar curvature is homeomorphic to $\bR^3$. Lastly, You and Zhang \cite[Theorem 1.1]{YouZhengTopRigidity} showed that a complete, contractible Riemannian 3-manifold with nonnegative scalar curvature is diffeomorphic to $\bR^3$. (See also the work of Chodosh--Lai--Xu \cite{ChodoshLaiXu3Manifolds} and Yan--Zhu \cite{YanZhuTopology3manifold}.)  We note that the condition of contractibility is necessary since the standard metric on the Riemannian product $\bS^2\times\bR$ has positive scalar curvature.

These results are related to the well-known open problem of Yau's \cite[Problem 27]{YauProblem} to classify 3-manifolds supporting complete Riemannian metrics with positive scalar curvature. Gromov \cite[\S 3.10.2]{GromovFourLectures} and Wang \cite{WanguniformlyPSC3}, independently, were able to make progress on this problem by showing an open, orientable 3-manifold supports a complete Riemannian metric with {uniform} positive scalar curvature if and only if it is an infinite connected sum of manifolds of the form $\bS^3/\Gamma_j$ and $\bS^2 \times \bS^1$, where $\Gamma_j$ is a finite subgroup of $\mathrm{SO}(4)$ acting freely on $\bS^3$. (See also Bessi\`eres--Besson--Maillot \cite{BessieresBessoMaillotRicciflow}, Bessi\`eres--Besson--Maillot--Marques \cite{BessieresBessoMaillotMarquesDeforming}, and Dong \cite{DongThreemanifolds}.) More recently, Chen \cite{ChenOptimal} showed that one can weaken the uniform positive scalar curvature hypothesis in Gromov's and Wang's result to positive scalar curvature with at most quadratic decay at infinity with constant $C > \frac{2}{3}$ and still have the same topological classification (see also Balacheff--Gil Moreno de Mora Sard\`a--Sabourau \cite{BalacheffMorenodeMoraSardàSabourauComplete}).

In dimension four, Chodosh, Máximo, and Mukherjee \cite{ChodoshMaximoMukherjeeComplete4} proved that if $M$ is a tame contractible 4-manifold and $M$ supports a complete Riemannian metric of uniform positive scalar curvature, then $M$ is homeomorphic to $\bR^4$. Moreover, if $M$ is the interior of a Mazur manifold\footnote{A Mazur manifold is a compact, contractible smooth 4-manifold with boundary admitting a handlebody decomposition with one 0-handle, one 1-handle, and one 2-handle.}, then the homeomorphism can be promoted to a diffeomorphism. In dimension five, the author \cite[Theorem A]{SweeneyPosCurv} showed that if $M$ is the interior of a compact, contractible 5-manifold with boundary $X$, such that $\pi_3(X, \partial X) = 0$ and $M$ supports a complete Riemannian metric of uniform positive scalar curvature, then $M$ is {diffeomorphic} to $\bR^5$. By the work of Chen \cite{ChenOptimal}, the uniform positive scalar curvature condition can be relaxed to positive scalar curvature with at most $C$-quadratic decay at infinity for some $C>\frac{n-1}{n}$ for both of these results. (See also the construction
criteria of Das \cite{DasConstructions} and the work of Wang--Zhu \cite{WangZhuQuadraticScalar47, 2WangZhuQuadraticScalar47}.) Finally, recall the recent work of Deng \cite[Theorem E]{DengConformallyFlat} mentioned earlier, which shows the following. Let $(M^n,g)$ be an open, simply connected Riemannian $n$-manifold, $n\geq 4$, supporting a complete, locally conformally flat metric with nonnegative scalar curvature, and let $\Phi: M\to\bS^n$ be the developing map. If $M$ is $\left\lfloor\frac{n-2}{2}\right\rfloor$-connected at infinity and $\bS^n\setminus\Phi(M)$ satisfies the small loops condition in $\bS^n$ then $M$ is homeomorphic to $\bR^n$.

\bigskip
The paper is organized as follows. In \Cref{prelims}, we collect important propositions from algebraic topology and Riemannian geometry, especially results about $\mu$-bubbles in complete manifolds with uniform positive BiRicci curvature and a lower bound on Ricci curvature, and in complete manifolds with positive BiRicci curvature with at most $C$-quadratic decay at infinity and nonnegative Ricci curvature. In \Cref{proofs}, we combine these tools to prove \Cref{t: main} and \Cref{t: main2}. The examples described in \Cref{t: example} are constructed in \Cref{s: examples proof}.

\section*{Acknowledgments}
 The author would also like to thank Demetre Kazaras for his many helpful conversations. The author would like to thank Jialong Deng for his helpful comments on their recent work on locally conformally flat manifolds and scalar curvature \cite{DengConformallyFlat} and its relation to this manuscript.

\section{Preliminaries}\label{prelims}

In this section, we collect relevant propositions from Riemannian geometry and topology that will be used to prove our main theorems.
\subsection{Algebraic Topology} Here, we recall some of the topological constraints on the boundaries of contractible manifolds and integral homology spheres, along with other useful algebraic topological facts. 

 \begin{prop}\label{homologysphere}
    Let $X^{n+1}$ be a compact, contractible $(n+1)$-manifold with boundary. Then $\partial X$ is an integral homology sphere, namely, $H_*(\partial X;\bZ)=H_*(\bS^n;\bZ)$.
\end{prop}
\begin{proof}
    As $X^{n+1}$ is contractible, $\pi_i(X)=0$, $H_i(X;\bZ)=0,$ and $H^i(X;\bZ)=0$ for all $i\geq 1$ and $H_0(X;\bZ)=\bZ$ and $H^0(X;\bZ)=\bZ$. Recall that the long exact sequence for the integral homology of $(X,\partial X)$ states
    \[
\cdots \to H_{i+1}(\partial X;\bZ) \to H_{i+1}(X;\bZ) \to H_{i+1}(X, \partial X;\bZ) \to H_{i}(\partial X;\bZ) \to H_{i}(X;\bZ) \to \cdots.
    \]
    On the other hand, Poincaré--Lefschetz duality says $H_{i+1}(X, \partial X;\bZ) \cong H^{n- i}(X;\bZ).$ For $1\leq i\leq n$, the long exact sequence becomes
    \[
    0\to H_{i+1}(X, \partial X;\bZ) \to H_{i}(\partial X;\bZ) \to0.
    \]
    Therefore, $H_{i}(\partial X;\bZ)\cong H_{i+1}(X, \partial X;\bZ)\cong H^{n - i}(X;\bZ)\cong 0$ for $1\leq i\leq n-1$ and  $H_{n}(\partial X;\bZ)\cong H_{n+1}(X, \partial X;\bZ)\cong H^{0}(X;\bZ)\cong \bZ$. 

    Since $X$ is connected, we know from Poincar\'e-Lefschetz duality that
    \[
    H_0(X,\partial X; \bZ)\cong H^{n+1}(X;\bZ)=0 \quad \text{  and  } \quad H_1(X,\partial X;\bZ)\cong H^n(X;\bZ) =0; 
    \]
    therefore,
    \[
    0 \to H_{0}(\partial X;\bZ) \to H_{0}(X;\bZ)\to 0.
    \]
    We see then that $H_{0}(\partial X;\bZ) \cong H_{0}(X;\bZ)\cong \bZ$.
\end{proof}

Next, we record a theorem of Sjerve \cite{SjerveHomologySpheres} on integral homology spheres that are covered by the sphere.
\begin{theorem}\label{coversofhomologyspheres}
    If $M^n$, $n\geq 3$, is an integral homology $n$-sphere which is covered by $\bS^n$, then either $\pi_1(M)=0$ or $n=3$ and $\pi_1(M)$ is the binary icosahedral group.
\end{theorem}
This means the Poincar\'{e} homology 3-sphere (which is $\bS^3/\Gamma$ where $\Gamma$ is the binary icosahedral group) is the only non-trivial integral homology $n$-sphere that is covered by the sphere.
\begin{lemma}\label{boundaryconnectedness}
    Let $X^{n+1}$ be a compact, contractible $(n+1)$-manifold with boundary. Then for $i\geq 1$, $\pi_{i+1}(X,\partial X)\cong\pi_i(\partial X).$
\end{lemma}
\begin{proof}
    As $X$ is contractible, we have for $i\geq 1$ that $\pi_i(X)=0$. Consider the long exact sequence of homotopy pairs $(X,\partial X)$.
    \[
        \cdots \to \pi_{i+1}(\partial X) \to \pi_{i+1}(X) \to \pi_{i+1}(X, \partial X) \to \pi_{i}(\partial X) \to \pi_{i}(X) \to \cdots.
    \]
  In particular, for $i\geq 1 $, $0 \to \pi_{i+1}(X,\partial X) \to \pi_{i}(\partial X) \to 0$. Therefore, $\pi_{i+1}(X,\partial X) \cong  \pi_{i}(\partial X)$.
\end{proof}

\begin{prop}\label{p: finitefundgroup}
    Let $M$ and $N$ be smooth, oriented, closed, connected, $n$-manifolds. If $f:M\to N$ is a smooth map of $\deg f =1$, then the induced homomorphism $f_*:\pi_1(M)\to \pi_1(N)$ is surjective.
\end{prop}
\begin{proof}
Consider the covering map $p:\hat{N}\to N$ corresponding to the subgroup $f_*(\pi_1(M))\subseteq \pi_1(N)$. As $f_*(\pi_1(M))\subseteq p_*(\pi_1(\hat{N}))=f_*(\pi_1(M))$, we lift $f$ to a map $\hat{f}:M\to \hat{N}$ such that $p\circ \hat{f} = f$. We will consider two cases.

 \noindent\underline{{\textsc{Case 1}}}: Assume $\hat{N}$ is compact, then $\deg p$ and $\deg\hat{f}$ are integers and $1=\deg f = (\deg p)\cdot (\deg\hat{f})$. Thus, $f_*(\pi_1(M))=p_*(\pi_1(\hat{N}))=\pi_1(N).$

 \noindent\underline{{\textsc{Case 2}}}: Assume $\hat{N}$ is noncompact. Choose a volume form $\omega$ on $N$ and note 
 \[
(\deg f)\int_N\omega= \int_M f^*\omega = \int_M \left(p\circ\hat{f}\right)^*\omega = \int_M\hat{f}^*(p^*\omega) =0,
 \]
 where the last equality follows from Stokes' Theorem since $\hat{N}$ is connected and non compact and so $p^*\omega$ is exact. This implies $\deg f =0$, which is a contradiction.
\end{proof}

We have the immediate corollary.
\begin{cor}\label{fintepi1}
    Let $M$ and $N$ be smooth, oriented, closed, connected $n$-manifolds and assume $\pi_1(M)$ is finite. If $f:M\to N$ is a smooth map of $\deg f =1$, then $\pi_1(N)$ is finite.
\end{cor}

\subsection{Riemannian Geometry}\label{RiemSubsection}
This section provides a brief overview of BiRicci curvature and $\mu$-bubbles. We collect some facts about $\mu$-bubbles in manifolds with certain positive BiRicci curvature conditions and certain lower bounds on Ricci curvature. For a more comprehensive treatment of $\mu$-bubbles, we refer the reader to Chodosh–Li \cite{ChodoshLiGeneralizedsoap}, Zhu \cite{ZhuWidth}, and Gromov \cite{GromovFourLectures}.

Let $(X^{n+1},g)$, $2\leq n\leq 6$, be a compact Riemannian $(n+1)$-manifold with boundary. Assume that $\partial X \neq \emptyset$ and that $\partial X$ has at least two boundary components. Now let $\partial X = \partial_-X \sqcup \partial_+X$ where both $\partial_+ X$ and $\partial_-X$ are nonempty and both $\partial_+ X$ and $\partial_-X$ are unions of connected components of the boundary. Now, fix a function $h\in C^\infty(\mathrm{int}(X))$ such that $h\to + \infty$ on $\partial_{+} X$ and $h\to -\infty$ on $\partial_{-} X$. Now, choose an open Caccioppoli set $\Omega_0$ with smooth boundary $\partial \Omega_0 \subset \mathrm{int}(X)$, $\partial_+X\subset \Omega_0$, and $\overline{\Omega}_0\cap \partial_-X=\emptyset$. Consider the following functional:

\begin{equation}\label{e: mububble}
    \mathcal{A}(\Omega) = \int_{\partial^* \Omega} d\mathcal{H}^{n} - \int_X \left(\chi_{_\Omega} - \chi_{_{\Omega_0}} \right)h\, d\mathcal{H}^{n+1}
\end{equation}
for all $\Omega \in \mathcal{C}:= \{\text{Caccioppoli sets } \Omega\subset X \text{ with } \Omega \triangle \Omega_0 \Subset \mathrm{int}(X)\}$. See \Cref{2d}.
\begin{figure}[H]
\begin{center}
\begin{tikzpicture}

    % Outer and inner circles (off-centered)
    \draw[thick, black] (-0.5,0) circle (2.5);
    \draw[thick, black] (0,0) circle (1);

    % Define the blue wiggly closed curve that fully encloses the inner circle
    \draw[thick, blue, smooth, samples=100, domain=0:360]
        plot ({1.5*cos(\x) + 0.1*sin(3*\x)}, {1.4*sin(\x) + 0.3*cos(2*\x)});

    % Label the blue curve
    \node at (0.8,-1.8) {\tiny \(\textcolor{blue}{\Omega_0}\)};

    % Shade the region inside the blue curve
   \fill[blue!30, opacity=0.6] 
        plot[smooth cycle, samples=100, domain=0:360]
        ({1.5*cos(\x) + 0.1*sin(3*\x)}, {1.4*sin(\x) + 0.3*cos(2*\x)});

    % Define the red wiggly closed curve inside the outer circle
    \draw[thick, red, smooth, samples=100, domain=0:360]
        plot ({1.7*cos(\x) + 0.1*sin(5*\x)}, {1.6*sin(\x) - 0.25*cos(3*\x)});
    % Shade the region inside the blue curve
   \fill[red!30, opacity=0.6] 
        plot[smooth cycle, samples=100, domain=0:360]
        ({1.7*cos(\x) + 0.1*sin(5*\x)}, {1.6*sin(\x) - 0.25*cos(3*\x)});
    % Label the red curve
    \node at (-1.9,0.5) {\tiny \(\textcolor{red}{\Omega}\)};

    %Shade region inside inner circle white
    \fill[white]  plot[smooth cycle, samples=100, domain=0:360] (0,0) circle (1);
    % Labels for the circles
    \node at (0,0.75) {\tiny \(\partial_+ X\)};
    \node at (-3.5,0) {\tiny \(\partial_- X\)};
\draw[thick, black] (0,0) circle (1);
\end{tikzpicture}
\end{center}
\caption{Schematic picture of a $\mu$-bubble} \label{2d}
\end{figure}
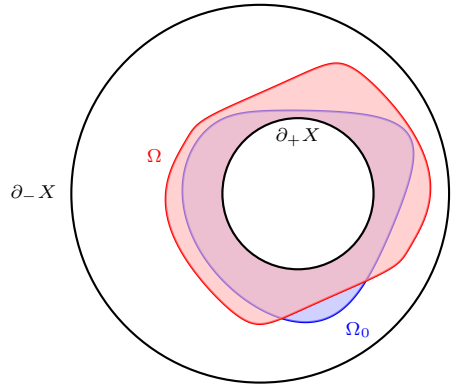

The existence and regularity of a minimizer of $\cA$ among all Caccioppoli
sets is claimed by Gromov \cite[Section 5.2]{GromovFourLectures} and rigorously carried out by Zhu \cite[Proposition 2.1]{ZhuWidth} (cf. \cite[Proposition 12]{ChodoshLiGeneralizedsoap}).

\begin{prop}\label{p: existence}
    For $2\leq n\leq 6$, there exists a smooth $(n+1)$-manifold $\Omega^{n+1}$ which minimizes $\cA$ on $\mathcal{C}$.
\end{prop}

\begin{definition}
    We call a minimizer $\Omega \in \mathcal{C}$ of $\mathcal{A}$ a $\mu$-bubble.
\end{definition}

Next, we record the first and second variations of $\cA$.
 
 \begin{lemma}\label{lemm:1st-var}
Let $\{\Omega_t\}_{t\in(-1,1)}$ be a smooth $1$-parameter family of regions in $\mathcal{C}$ with $\Omega_0 = \Omega$, $\partial \Omega_t=\Sigma_t$, and normal speed $\psi\in C^\infty_c(\Sigma)$ at $t=0$, then 
\[\frac{d}{dt}\cA (\Omega_t)\big|_{t=0}=\int_{\Sigma} (H - h)\psi  \, dA_g,\]
where $\Sigma=\partial \Omega$ and $H$ is the mean curvature of $\Sigma$ which is computed with respect to $\nu$, the outward pointing unit normal to $\partial \Omega=\Sigma$. In particular, a $\mu$-bubble $\Omega$ satisfies a prescribed mean curvature condition
\[
H = h
\]
along $\partial \Omega$. 
\end{lemma}

\begin{lemma}\label{lemm:2nd-var}
Consider a $\mu$-bubble $\Omega$ with $\partial \Omega = \Sigma$. Assume that $\{\Omega_t\}_{t\in(-1,1)}$ is a smooth $1$-parameter family of regions in $\mathcal{C}$ with $\Omega_0 = \Omega$ and normal speed $\psi\in C^\infty_c(\Sigma)$ at $t=0$, then $\frac{d^2}{dt^2}\big|_{t=0}(\cA(\Omega_t))\ge 0$, where 
\[
\frac{d^2}{dt^2}\big|_{t=0}(\cA(\Omega_t))= \int_\Sigma \left(|\nabla_\Sigma \psi|^2- (\Ric_X(\nu,\nu)+|A|^2+\langle \nabla h, \nu \rangle )\psi^2 \right)\, dA_g.
\]
\end{lemma} 

The next proposition is a version of a separation result of Gromov \cite[Section 3.7]{GromovFourLectures} for complete manifolds with $\BiRic\geq \kappa>0$ and $\Ric\geq -K$.

\begin{prop}\label{separationthm}
    Fix a constant $K\geq 0$. Let $(X^{n+1},g)$, $n=5$, be a compact Riemannian $(n+1)$-manifold with boundary. Assume that $\partial X \neq \emptyset$ and that $\partial X$ has at least two boundary components. Let $\partial X = \partial_-X \sqcup \partial_+X$ where $\partial_\pm X \neq \emptyset$ are unions of connected components of the boundary. Assume that $\BiRic\geq \kappa>0$ and $\Ric\geq -K$. Then there is a constant $C(\kappa,K)>\max\left\{128\pi,\frac{32K^2\pi}{\kappa^2},\frac{8\pi}{\kappa}\right\}$ such that if $d(\partial_-X,\partial_+ X)> C$, there exists a smooth embedded closed 2-sided hypersurface $\Sigma\subset \mathrm{int}(X) $ that separates $\partial_-X$ from $\partial_+X$ and every connected component of $\Sigma$ has finite fundamental group.
\end{prop}
For the sake of completeness, we provide a quick sketch of the proof of \Cref{separationthm}. Furthermore, we have stated the proposition in terms of $n$ so that it is clear in the proof where $n=5$ is essential.
\begin{proof}
Let $D=d\left(\partial_-X,\partial_+ X\right)>C$, let $0<\eps< \min\left\{\frac{D-C}{2},\frac{1}{4}\right\}$, let $C'>0$ be such that $\eps<C'<C'+C<D-\eps$. By \cite[\S 2]{GreeneWuCinftyapprox}, there exists a smooth function $\rho:X\to\bR$ such that $\left|\rho(x)-d(x,\partial_-X)\right|<\eps$ and $|\Lip(\rho)|\leq 1+\eps$ for all $x\in X$. We can perturb further and assume that $C'$ and $C'+C$ are regular values of $\rho$. Define the band
\[
X_0=\{x\in X: C'\leq \rho(x)\leq C'+C\}\Subset \mathrm{int}(X).
\]

\noindent Finally, define the smooth function $h:\mathrm{int}(X_0)\to \bR$
\begin{equation}
    h(x):=\tan\left(\frac{\rho(x)-C'-\frac{1}{2}C}{C}\pi\right)
\end{equation}
and note that $h\to\pm\infty$ on $\partial_\pm X_0$, respectively.%  (see \Cref{hgraph}).

With respect to this $h$, we can find a $\mu$-bubble ${\Omega}$ with boundary $\partial^*\Omega=\Sigma$ by using \Cref{p: existence}.

By \Cref{lemm:2nd-var}, for every $\psi\in C^\infty_c(\Sigma)$, 
\begin{align}
\frac{d^2}{dt^2}\big|_{t=0}(\cA(\Omega_t))= \int_\Sigma \left(|\nabla_\Sigma \psi|^2- (\Ric_X(\nu,\nu)+|A|^2+\langle \nabla h, \nu \rangle )\psi^2 \right)\, dA_g\geq 0.
\end{align}
Next, we choose $e_1$ so that $\Ric_{\Sigma,p}(e_1,e_1)=\inf\{\Ric_{\Sigma,p}(v,v)|v\in T_p\Sigma, |v|=1\}:=\lambda_{\Ric_\Sigma}(p)$, let $\gamma>1$, and compute:
\begin{equation}\label{eq: biricci computation}
\begin{split}
    \gamma\left(\Ric_X(\nu,\nu) +|A|^2\right) +\lambda_{\Ric_\Sigma} &= \BiRic(\nu,e_1) + (\gamma -1)\Ric_X(\nu,\nu)\\
    &\qquad+\gamma|A|^2-\sum_{j=1}^nA(e_1,e_j)^2 + hA(e_1,e_1)\\
    &\geq \BiRic(p) - (\gamma-1)K + (\gamma-1)A(e_1,e_1)^2\\
    &\qquad+\frac{\gamma}{n-1}(h-A(e_1,e_1))^2+hA(e_1,e_1)\\
    &\geq \kappa-(\gamma -1)K + \frac{4\gamma^2-(n-1)}{4(n\gamma-(n-1))}h^2\\
    &\qquad+ \frac{n\gamma-(n-1)}{n-1}\left(A(e_1,e_1)+\frac{n-1-2\gamma}{2(n\gamma-(n-1))}h\right)^2\\
    &\geq \kappa-(\gamma -1)K + \frac{4\gamma^2-(n-1)}{4(n\gamma-(n-1))}h^2,
\end{split}
\end{equation}
where the first equality is the Gauss--Codazzi equations, the second is an application of Cauchy--Schwarz in a manner similar to the Kato inequality to the symmetric matrix $A$, and the third inequality comes from completing the square.
Therefore, from the stability inequality we have,
\begin{align}\label{e: equation}
\begin{split}
  \int_\Sigma \gamma|\nabla_\Sigma \psi|^2 &+\lambda_{\Ric_\Sigma}\psi^2\,dA_g \\
  &\geq \int_\Sigma \left(\gamma\Ric_X(\nu,\nu)+\gamma|A|^2+\lambda_{\Ric_\Sigma}+\gamma\langle \nabla h, \nu \rangle \right)\psi^2 \, dA_g\\
  &\geq \int_\Sigma \left(\kappa-(\gamma -1)K + \frac{4\gamma^2-(n-1)}{4(n\gamma-(n-1))}h^2-\gamma|\nabla h|\right)\psi^2 \, dA_g.
\end{split}
\end{align}
Now, if we choose $\gamma > \frac{\sqrt{n-1}}{2}$, the coefficient of $h^2$ is positive. 

In the rest of the proof we use that $n=5$. The importance of this dimensional restriction comes from the fact that $\gamma>1=\frac{\sqrt{5-1}}{2}$.

As $n=5$, we can choose $\gamma=1+\eta$, where $\eta =\min\left\{\frac{1}{3},\frac{\kappa}{2K}\right\}$ if $K>0$ and $\eta=\
\frac{1}{3}$ if $K=0$. We note 
\begin{equation}\label{eq: lowerboundh}
    \kappa-(\gamma-1) K + \frac{4\gamma^2-(n-1)}{4(n\gamma-(n-1))}h^2-\gamma|\nabla h|\geq \frac{\kappa}{2} + \frac{\eta^2}{2} h^2 -2|\nabla h|.
\end{equation}

Now, we compute:
\begin{align*}
 \frac{\eta^2}{2}h^2 -2|\nabla h| &=\frac{\eta^2}{2}\tan^2\left(\frac{\rho(x)-C'-\frac{1}{2}C}{C}\pi\right) -\frac{2|\nabla\rho(x)|\pi}{C}\sec^2\left(\frac{\rho(x)-C'-\frac{1}{2}C}{C}\pi\right)\\
     &\geq \frac{\eta^2}{2}\tan^2\left(\frac{\rho(x)-C'-\frac{1}{2}C}{C}\pi\right) -\frac{2 \pi}{C}(1+\eps)\sec^2\left(\frac{\rho(x)-C'-\frac{1}{2}C}{C}\pi\right)\\
     &\geq \frac{2\pi}{C}(1+\eps)\tan^2\left(\frac{\rho(x)-C'-\frac{1}{2}C}{C}\pi\right) -\frac{2\pi}{C}(1+\eps) \sec^2\left(\frac{\rho(x)-C'-\frac{1}{2}C}{C}\pi\right)\\
     &\geq -\frac{4\pi}{C},
\end{align*}
where the penultimate inequality follows since $C> \frac{8\pi}{\eta^2}$ and $0<\eps<\frac{1}{4}$ and the final inequality follows by the trigonometric identity $\tan^2(\theta)-\sec^2(\theta)=-1$ and $0<\eps<\frac{1}{4}$.

Therefore, by \eqref{e: equation} and \eqref{eq: lowerboundh}, we have
\[
    \int_{\Sigma} \left(\gamma|\nabla_\Sigma \psi|^2+\lambda_{\Ric_\Sigma}\psi^2 \right)\, dA_g\geq \int_{\Sigma} \left(\frac{\kappa}{2} -\frac{4\pi}{C} \right)\psi^2\, dA_g.
\]
Since $C>\frac{8\pi}{\kappa}$, we conclude that every connected component $\Sigma'$ of $\Sigma$ has positive spectral Ricci curvature. Note $\Sigma'$ is a closed manifold and $\gamma\leq\frac{4}{3}$ and so by \cite[Theorem 1]{AntonelliXuNewSpectralBishopGromov}, we have that $\pi_1(\Sigma')$ is finite.
\end{proof}

The next proposition is a version of an exhaustion result of Gromov \cite[Section 3.7]{GromovFourLectures} (cf.  \cite[Proposition 1.8]{ChenOptimal}) for complete manifolds with positive BiRicci curvature with $C$-quadratic decay at infinity.
   
\begin{prop}\label{exhaustionthm}
     Let $(X^{n+1},g)$, $n=5,6$, be a complete,  noncompact, orientable Riemannian $(n+1)$-manifold. Assume that $g$ has positive BiRicci curvature with at most $C$-quadratic decay for some $C>C_{n+1}$, where $C_6=\frac{32}{21}$ and $C_7=\frac{25}{8}$ and $\Ric\geq 0$. Then $X$ admits an exhaustion $K_1\subset K_2\subset \cdots\subset X$ such that $X=\cup_{i=1}^\infty K_i$, each $K_i$ is a compact codimension-0 smooth submanifold with smooth boundary  $\partial K_i:=\Sigma_i$, and every connected component of $\Sigma_i$ is a compact orientable $n$-manifold with finite fundamental group.
\end{prop}
For the sake of completeness, we provide a quick proof of \Cref{exhaustionthm} which is based on \cite[Proposition 1.8]{ChenOptimal}.
\begin{proof}
  Fix a base point $x_0\in X$, $\gamma = \frac{n-1}{n-2}$, $\alpha=\frac{4\gamma^2-(n-1)}{4(n\gamma-(n-1))} = \frac{-n^2+8n-8}{8(n-2)}$, and denote by $r(y)=d(x_0,y)$, the distance function to $x_0$. By assumption, there exists a $D_0>1$ such that if $r(y)>D_0$, then $\BiRic(y)\geq \frac{C}{r(y)^2}$. Let $C_{n+1}= \frac{\gamma^2}{4 \alpha}<\tilde{C}<C$.  Now let $\tilde{\mu}=\sqrt{\frac{\alpha\tilde{C}}{\gamma^2}-\frac{1}{4}}>0$, and choose $\eps=\frac{\sqrt{\frac{C}{\tilde{C}}}-1}{\sqrt{\frac{C}{\tilde{C}}}+1}\in(0,1)$. Thus,
 \[
 \frac{(1+\eps)^2D_0^2}{(D_0-\eps)^2} = \left(\frac{1+\eps}{1-\frac{\eps}{D_0}}\right)^2 < \left(\frac{1+\eps}{1-\eps}\right)^2 = \frac{C}{\tilde{C}}.
 \]

 Note that $r\mapsto \frac{(1+\eps)^2r^2}{(r-\eps)^2}$ is a decreasing function for all $r>\eps$. Thus, for all $r\geq D_0$, we have
 \begin{equation}\label{eq: Cestimate}
      \frac{\tilde{C}(1+\eps)^2}{(r-\eps)^2}< \frac{C}{r^2}.
 \end{equation}

 By \cite[\S 2]{GreeneWuCinftyapprox}, there exists a smooth function $\rho:X\to\bR$ such that $|\rho(x)-r(x)|<\eps$ and $|\Lip(\rho)|\leq 1+\eps$ for all $x\in X$. We can perturb further and assume that for all $k\in\bN$, $(1+\eps)\exp\left(\tilde{\mu}^{-1}\left(k\pi-\frac{\pi}{2}\right)\right)$ are regular values. Finally, define the function $h:X\to \bR$
\begin{equation}
    h(x):=-\frac{\gamma(1+\eps)}{\alpha\rho(x)}\left(-\frac{1}{2} + \tilde{\mu} \tan\left(\tilde{\mu}\log\left(\frac{\rho(x)}{1+\eps}\right)\right)\right).
\end{equation}

Take $k_1>0$ to be an integer such that $\exp\left(\tilde{\mu}^{-1}\left(k_1\pi-\frac{\pi}{2}\right)\right)>D_0+1$ and define
\[
X_0=\left\{x\in X|\rho(x)<(1+\eps)\exp\left(\tilde{\mu}^{-1}\left(k_1\pi-\frac{\pi}{2}\right)\right)\right\}
\]
and 
\[
X_1=\left\{x\in X|(1+\eps)\exp\left(\tilde{\mu}^{-1}\left(k_1\pi-\frac{\pi}{2}\right)\right)\leq\rho(x)\leq(1+\eps)\exp\left(\tilde{\mu}^{-1}\left(k_1\pi+\frac{\pi}{2}\right)\right)\right\}
\]
with boundaries
\[
\partial_+X_1 = \left\{x\in X|\rho(x)=(1+\eps)\exp\left(\tilde{\mu}^{-1}\left(k_1\pi-\frac{\pi}{2}\right)\right)\right\}
\]
and 
\[
\partial_-X_1=
\left\{x\in X|\rho(x)=(1+\eps)\exp\left(\tilde{\mu}^{-1}\left(k_1\pi+\frac{\pi}{2}\right)\right)\right\}.
\]

Note that $h\to\pm\infty$ on $\partial_\pm X_1$, respectively and on $X_1$, $\rho>D_0$ and $r>D_0$. Let $L_1\in \left((1+\eps)\exp\left(\tilde{\mu}^{-1}\left(k_1\pi-\frac{\pi}{2}\right)\right),(1+\eps)\exp\left(\tilde{\mu}^{-1}\left(k_1\pi+\frac{\pi}{2}\right)\right)\right)$ be a regular value of $\rho$. Fix a smooth reference region 
\[
\Omega_0 = \left\{x\in X| (1+\eps)\exp\left(\tilde{\mu}^{-1}\left(k_1\pi-\frac{\pi}{2}\right)\right)\leq\rho(x)\leq L_1\right\}.
\]
With respect to this $h$, we can find a $\mu$-bubble ${\Omega_1}$ with boundary $\partial^*\Omega_1\cap \mathrm{int}(X_1)=\Sigma_1$ by using \Cref{p: existence}. Note that $\Sigma_1$ is orientable because $X$ is and let $K_1=X_0\cup \Omega_1$. Then $K_1$ is a smooth compact region and $\partial K_1=\Sigma_1$.

By \Cref{lemm:2nd-var}, \eqref{eq: biricci computation}, and \eqref{e: equation}, we have for every $\psi\in C^\infty_c(\Sigma_1)$, 
\begin{align}\label{e: spectral ricci 2}
\begin{split}
  \int_{\Sigma_1} \gamma|\nabla_{\Sigma_1} \psi|^2 &+\lambda_{\Ric_{\Sigma_1}}\psi^2\,dA_g 
  \geq \int_{\Sigma_1} \left(\frac{C}{r^2} + {\alpha}h^2-\gamma|\nabla h|\right)\psi^2 \, dA_g.
\end{split}
\end{align}

Now, we first compute:
\begin{align*}
    |\nabla h| &= \frac{\gamma(1+\eps)|\nabla\rho(x)|}{\alpha\rho(x)^2}\left|\frac{1}{2}-\tilde{\mu}\tan\left(\tilde{\mu}\log\left( \frac{\rho(x)}{1+\eps}\right)\right) + \tilde{\mu}^2\left(1+\tan^2\left(\tilde{\mu}\log\left( \frac{\rho(x)}{1+\eps}\right)\right)\right)\right|\\
    &=\frac{\gamma(1+\eps)|\nabla\rho(x)|}{\alpha\rho(x)^2}\left(\left(-\frac{1}{2}+\tilde{\mu}\tan\left(\tilde{\mu}\log\left(\frac{\rho(x)}{1+\eps}\right)\right)\right)^2+\frac{\alpha\tilde{C}}{\gamma^2}\right).
\end{align*}
Using this, we have
\begin{align*}
     \frac{C}{r^2} +& \alpha h^2 -\gamma| \nabla h|\\
     &=\frac{C}{r^2}+\alpha h^2 - \frac{\gamma^2(1+\eps)|\nabla\rho(x)|}{\alpha\rho(x)^2}\left(\left(-\frac{1}{2}+\tilde{\mu}\tan\left(\tilde{\mu}\log\left(\frac{\rho(x)}{1+\eps}\right)\right)\right)^2+\frac{\alpha\tilde{C}}{\gamma^2}\right) \\
     &\geq \frac{C}{r^2}+\alpha h^2 - \frac{\gamma^2(1+\eps)^2}{\alpha\rho(x)^2}\left(-\frac{1}{2}+\tilde{\mu}\tan\left(\tilde{\mu}\log\left(\frac{\rho(x)}{1+\eps}\right)\right)\right)^2-\frac{\tilde{C}(1+\eps)^2}{\rho(x)^2} \\
     &\geq \frac{C}{r(x)^2}-\frac{\tilde{C}(1+\eps)^2}{(r(x)-\eps)^2} \\
     &>0,
\end{align*}
where the last inequality follows from \eqref{eq: Cestimate}. Therefore by \eqref{e: spectral ricci 2}, we see that every connected component $\Sigma'$ of $\Sigma_1$ has positive spectral Ricci curvature. As $\Sigma'$ is a closed manifold and $\gamma=\frac{n-1}{n-2}$, by \cite[Theorem 1]{AntonelliXuNewSpectralBishopGromov}, we have that $\pi_1(\Sigma')$ is finite.

Lastly, assume that $K_1\subset\cdots  \subset K_{i-1}$ have been chosen as in the statement of the proposition. Let $k_i=k_1+2(i-1)$. Then we can find $X_i$, $\Omega_i$, and $\Sigma_i$ in the same manner as $X_1$, $\Omega_1$, and $\Sigma_1$. Let $K_i=\{x\in X:\rho(x) < (1+\eps)\exp\left(\tilde{\mu}^{-1}\left(k_i\pi-\frac{\pi}{2}\right)\right)\} \cup \Omega_i$. The same computation shows that every connected component of $\Sigma_i$ has finite fundamental group. \end{proof}

\begin{remark}
    Here we provide some details about the choice of $C_6$ and $C_7$. Within the proof, we note that $\gamma$ needs to be chosen strictly larger than $\frac{\sqrt{n-1}}{2}$ so that the coefficient of $h^2$ is strictly positive (see \eqref{e: equation}). In order to apply \cite[Theorem 1]{AntonelliXuNewSpectralBishopGromov}, we must choose $\gamma\le \frac{n-1}{n-2}$. Now, note for $\gamma\in\left(\frac{\sqrt{n-1}}{2},\frac{n-1}{n-2}\right]$ and $\alpha:=\alpha(\gamma)=\frac{4\gamma^2-(n-1)}{4(n\gamma-(n-1))}$, we have that  $\gamma\mapsto \frac{\gamma^2}{4\alpha(\gamma)}$ is a decreasing function. Therefore, the smallest value of $\frac{\gamma^2}{4\alpha(\gamma)}$ we can choose is when $\gamma=\frac{n-1}{n-2}$, which corresponds to $C_{6}=\frac{32}{21}$ and $C_7=\frac{25}{8}$.
\end{remark}

 We can apply \Cref{separationthm} to show the following corollary.

\begin{cor}\label{BiRicFiniteFundGroup}
    Let $X^{6}$ be a compact, contractible $6$-manifold with boundary such that the interior of $X$ supports a complete Riemannian metric with $\BiRic\geq \kappa>0$ and $\Ric\geq-K$ for some $K\geq0$. Then $\pi_1(\partial X)$ is finite.
\end{cor}
\begin{proof}
   By \Cref{homologysphere},  we have that $\partial X$ is an integral homology sphere; therefore, $\partial X$ is connected and orientable. By the collar neighborhood theorem, there is a closed set $U\Subset \mathrm{int}(X)$, diffeomorphic to $\partial X \times [-1,1]$. Moreover, by completeness of the metric, we may assume $d\left(\partial X \times \{-1\},\partial X\times\{1\}\right)>C(\kappa, K)$, where $C(\kappa, K)$ is the constant in the hypothesis of \Cref{separationthm}.  Therefore, from \Cref{separationthm}, there is a smooth embedded 2-sided hypersurface $\Sigma$ in $U$ separating $\partial X \times \{-1\}$ from $\partial X \times \{1\}$ and every connected component of $\Sigma$ has finite fundamental group.

   Consider now the projection map $\pi:U\to \partial X$ and its restriction $\pi|_{ \Sigma}:\Sigma\to \partial X$. If $\omega$ is the volume form on $\partial X$, then $\int_{\Sigma}(\pi|_{\Sigma})^*\omega = \pm 1 \int_{\partial X}\omega$. (One can see this since $\Sigma\subset U$ separates $\partial X \times \{-1\}$ from $\partial X \times \{1\}$, which implies there is a $6$-manifold $W\subseteq U$ with $\partial W= \Sigma \sqcup \partial X$ and applying Stokes' Theorem.) Therefore, $\deg \pi_\Sigma =\pm1$. And so there is a degree one map from a connected component $\Sigma'$ of $\Sigma$ to $\partial X$ which implies, by \Cref{fintepi1}, that $\pi_1(\partial X)$ is finite.
   \end{proof}

We can apply \Cref{exhaustionthm} to show the following corollary.

\begin{cor}\label{BiRicFiniteFundGroup2}
    Let $X^{n}$, $n\in\{6,7\}$ be a compact, contractible $n$-manifold with boundary such that the interior of $X$ supports a complete Riemannian metric with positive BiRicci curvature with at most $C$-quadratic decay for some $C>C_{n}$, where $C_6=\frac{32}{21}$ and $C_7=\frac{25}{8}$ and $\Ric\geq 0$. Then the boundary $\partial X$ has a finite fundamental group.
\end{cor}
\begin{proof}
   By \Cref{homologysphere},  we have that $\partial X$ is an integral homology sphere; therefore, $\partial X$ is connected and orientable. Let $V\subset \mathrm{int}(X)$ be an open subset such that 
   \[
   \mathrm{int}(X) \setminus V \cong \partial X\times [0,\infty)
   \]
   Therefore, from \Cref{exhaustionthm}, we can choose a sufficiently large $i$ so that $\overline{V}\Subset{K_i}$, and so $\Sigma=\partial K_i\subset \partial X \times [0,\infty)$ is a smooth embedded 2-sided hypersurface $\Sigma\subset\partial X \times [a,b]\subset \partial X \times [0,\infty)$ separating $\partial X \times \{a\}$ from $\partial X \times \{b\}$. Also, by \Cref{exhaustionthm}, every connected component of $\Sigma$ has a finite fundamental group.

   Consider now the projection map $\pi:\partial X \times [a,b]\to \partial X$ and its restriction $\pi|_{ \Sigma}:\Sigma\to \partial X$. If $\omega$ is the volume form on $\partial X$, then $\int_{\Sigma}(\pi|_{\Sigma})^*\omega = \pm 1 \int_{\partial X}\omega$. (One can see this since $\Sigma$ separates $\partial X \times \{a\}$ from $\partial X \times \{b\}$, which implies there is an $n$-manifold $W\subseteq \partial X\times [a,b]$ with $\partial W= \Sigma \sqcup \partial X$ and applying Stokes' Theorem.) Therefore, $\deg \pi_\Sigma =\pm1$. And so there is a degree one map from a connected component $\Sigma'$ of $\Sigma$ to $\partial X$, which implies by \Cref{fintepi1} that $\pi_1(\partial X)$ is finite.
   \end{proof}

%

% \begin{cor}\label{BiRicFiniteFundGroup}
%     Let $X^{n+1}$  be a compact, contractible $n$-manifold with boundary such that the interior of $X$ supports a complete Riemannian metric with positive BiRicci curvature with at most $C$-quadratic decay for some $C>\frac{n^2}{4}$.  Then the boundary $\partial X$ has a finite fundamental group.
% \end{cor}
% \begin{proof}
%    By \Cref{homologysphere},  we have that $\partial X$ is an integral homology sphere; therefore, $\partial X$ is connected and orientable. By the collar neighborhood theorem, there is a closed set $U$ that contains the boundary, diffeomorphic to $\partial X \times [-1,1]$. Therefore, from \Cref{exhaustionthm}, there is a smooth embedded 2-sided hypersurface $\Sigma$ in $U$ separating $\partial X \times \{-1\}$ from $\partial X \times \{1\}$ and $\Sigma$ has finite fundamental group.

%    Consider now the projection map $\pi:U\to \partial X$ and its restriction $\pi|_{ \Sigma}:\Sigma\to \partial X$. If $\omega$ is the volume form on $\partial X$, then $\int_{\Sigma}(\pi|_{\Sigma})^*\omega = \pm 1 \int_{\partial X}\omega$. (One can see this since $\Sigma\subset U$ separates $\partial X \times \{-1\}$ from $\partial X \times \{1\}$, which implies there is an $(n+1)$-manifold $W\subseteq U$ with $\partial W= \Sigma \sqcup \partial X$ and applying Stokes' Theorem.) Therefore, there is a degree one map from $\Sigma$ to $\partial X$ which implies, by \Cref{fintepi1}, that $\pi_1(\partial X)$ is finite.
%    \end{proof}

\section{Proof of the main theorems}\label{proofs}
In this section, we combine the above algebraic topological facts and $\mu$-bubble results to prove the main results. %For ease of reading, we will prove \Cref{t: main} and \Cref{t: main2} separately, even though their proofs are similar. 

\begin{proof}[Proof of \Cref{t: main}]
    By assumption and \Cref{boundaryconnectedness}, we have $\pi_2(\partial X)=0$, by \Cref{homologysphere}, we know that $\partial X$ is a homology sphere, and by \Cref{BiRicFiniteFundGroup}, we know that $\pi_1(\partial X)$ is finite. As $\partial X$ is a closed, connected manifold and $\pi_1(\partial X)$ is finite, we know that the universal cover $\widetilde{\partial X}$ is also an orientable, closed, connected manifold and so $H_0(\widetilde{\partial X};\bZ)=\bZ$ and $H_5(\widetilde{\partial X};\bZ)=\bZ$. Note, $\pi_1(\widetilde{\partial X})=0$ and $\pi_2(\widetilde{\partial X})=0$ and so by the Hurewicz Theorem we know $H_1(\widetilde{\partial X};\bZ)=0$ and  $H_2(\widetilde{\partial X};\bZ)=0$ and $\pi_3(\widetilde{\partial X})\cong H_3(\widetilde{\partial X};\bZ)$. Now by the universal coefficient theorem and Poincar\'e duality, we have that $ 0=\mathrm{Hom}(H_1(\widetilde{\partial X};\bZ),\bZ)\cong H^1(\widetilde{\partial X};\bZ) \cong H_4(\widetilde{\partial X};\bZ)$. Also, by Poincar\'e duality and the universal coefficient theorem, $H_3(\widetilde{\partial X};\bZ)\cong H^2(\widetilde{\partial X},\bZ)\cong\mathrm{Hom}(H_2(\widetilde{\partial X};\bZ),\bZ)\oplus\mathrm{Ext}(H_1(\widetilde{\partial X};\bZ),\bZ)=0$. 
    
    Therefore, $\widetilde{\partial X}$ is a simply connected homology sphere and so by \cite[\S 9 Proposition B]{MilnorLectures} (see also Smale \cite{SmaleGeneralized, SmaleN}), we have that $\widetilde{\partial X}$ is diffeomorphic to $\bS^5$. By \Cref{coversofhomologyspheres}, $\partial X$ is a simply connected homology sphere. Finally, by  \cite[\S 9 Proposition A]{MilnorLectures}, $X$ is diffeomorphic to the $6$-disk and so $M$ is diffeomorphic to $\bR^6$ by Stallings \cite{StallingsThepiecewise}.
\end{proof}

Now, we prove the case $n=7$ in \Cref{t: main2}.
\begin{proof}[Proof of \Cref{t: main2} when n=7]
    By assumption and \Cref{boundaryconnectedness}, we have $\pi_2(\partial X)=0$. By \Cref{homologysphere}, we know that $\partial X$ is a homology sphere, and by \Cref{BiRicFiniteFundGroup2}, we know that $|\pi_1(\partial X)|=d$ is finite. As $\partial X$ is a closed, connected manifold and $\pi_1(\partial X)$ is finite, we know that the universal cover $\widetilde{\partial X}$ is also an orientable, closed, connected manifold and so $H_0(\widetilde{\partial X};\bZ)=\bZ$ and $H_6(\widetilde{\partial X};\bZ)=\bZ$. Note, $\pi_1(\widetilde{\partial X})=0$ and $\pi_2(\widetilde{\partial X})=0$ so by the Hurewicz Theorem we know $H_1(\widetilde{\partial X};\bZ)=0$ and $H_2(\widetilde{\partial X};\bZ)=0$. Now, by Poincar\'e duality and the universal coefficient theorem, we have that $H_4(\widetilde{\partial X};\bZ)\cong H^2(\widetilde{\partial X},\bZ)\cong\mathrm{Hom}(H_2(\widetilde{\partial X};\bZ),\bZ)\oplus\mathrm{Ext}(H_1(\widetilde{\partial X};\bZ),\bZ)=0$. Again, by Poincar\'e duality and the universal coefficient theorem, we have that $H_5(\widetilde{\partial X};\bZ)\cong H^1(\widetilde{\partial X},\bZ)\cong \mathrm{Hom}(H_1(\widetilde{\partial X};\bZ),\bZ)\oplus\mathrm{Ext}(H_0(\widetilde{\partial X};\bZ),\bZ)=0$. Now, we compute the Euler characteristic
    \[
    2d=\chi(\partial X)\cdot d = \chi(\widetilde{\partial X}) = \sum_{j=0}^6 (-1)^j\mathrm{rank}H_j(\widetilde{\partial X})=2-\mathrm{rank}H_3(\widetilde{\partial X}).
    \]
    Rearranging this, we see
    \[
    0\leq \mathrm{rank}H_3(\widetilde{\partial X}) = 2(1-d).
    \]
    As $d$ is an integer greater than or equal to one, we conclude $d=1$.
    
    Therefore, $\pi_1(\partial X)$ has one element. Thus, $\partial X$ is a simply connected homology sphere and so by \cite[\S 9 Proposition B]{MilnorLectures} (see also Smale \cite{SmaleGeneralized, SmaleN}), we have that ${\partial X}$ is diffeomorphic to $\bS^6$. By  \cite[\S 9 Proposition A]{MilnorLectures}, $X$ is diffeomorphic to the $7$-disk and so $M$ is diffeomorphic to $\bR^7$ by Stallings \cite{StallingsThepiecewise}.
\end{proof}

\begin{remark}
    The proof of \Cref{t: main2} when $n=6$ follows from an identical argument as the proof of \Cref{t: main} but with \Cref{BiRicFiniteFundGroup} replaced with \Cref{BiRicFiniteFundGroup2}.
\end{remark}

\section{Construction of sufficiently connected contractible manifolds}\label{s: examples proof}

Before we construct the example, we start with some background and propositions. In \cite{HigmanAFinitelyGenerated}, Higman constructed a nontrivial group $H$ with many interesting properties. $H$ has the following presentation
\begin{equation}\label{eq: presentation}
    \langle w,x,y,z| wxw^{-1}x^{-2}, xyx^{-1}y^{-2}, yzy^{-1}z^{-2},zwz^{-1}w^{-2}\rangle.
\end{equation}

It was then observed by Dyer and Vasquez \cite[\S 4]{DyerVasquez} that the canonical two-dimensional cell-complex $K$ of this presentation also has interesting topological properties. (See also \cite{KornevHomologySpheres} for a survey about homology spheres and further discussion of $H$.)

\begin{prop}[{\cite[\S 4]{DyerVasquez}}]\label{p: Higman}
    The canonical two-dimensional cell-complex $K$ associated to the presentation \eqref{eq: presentation} of $H$ is made of one vertex, four 1-cells, and four 2-cells and enjoys the following properties:
    \begin{enumerate}
        \item $\pi_1(K)=H$,
        \item $\pi_i(K)=0$, $i\geq 2$,
        \item $H_i(K;\bZ)=0$, $i\geq 1$,
        \item $H^i(K;\bZ)=0$, $i\geq 1$.
    \end{enumerate}
\end{prop}

Next we will describe a construction of Newman \cite{NewmanBoundaries} of tame contractible manifolds (see also \cite[Example 0.10.3]{DavermanVenemaEmbeddings} and \cite[\S 2 and \S 3]{GuilbaultSomeCompactContractible}) applied to our two-complex $K$. The next proposition will prove \Cref{t: example}.

\begin{prop}
    For each $n\geq 6$, there exists a compact, contractible manifold with boundary $N^n$ such that $\pi_1(\partial N)=H$, $\pi_i(N,\partial N)=0$, $3\leq i\leq n-3$. Moreover, $\mathrm{int}(N)$ admits a complete metric with uniform positive scalar curvature and Ricci curvature bounded below.
\end{prop}
\begin{proof}
    This is essentially the proof from \cite[Example 0.10.3]{DavermanVenemaEmbeddings}. 

    Embed the canonical two-dimensional cell-complex $K$ associated to the presentation \eqref{eq: presentation} of $H$ into $\bS^n$, $n\geq 6$. By \cite[Theorem 1]{HirschSmoothRegular}, we can find a smooth closed regular neighborhood $\overline{U}$ of $K$ and let $N:=\bS^n\setminus \mathrm{int}(\overline{U})$. 
    
    Next, we will show that $N$ has the desired topological properties. To do this we will first show that $\pi_1(N)=0$ and $H_i(N)=0$, $i\geq 1$.
    
    Note that $N$ is a deformation retract of $S^n\setminus K$ and $K$ is a deformation retract of $\overline{U}$; thus, $S^n\setminus K$ is homotopy equivalent to $N$ and $K$ is homotopy equivalent to $\overline{U}$. By general position considerations, $0=\pi_1(S^n\setminus K)\cong\pi_1(N)$ and note this implies $H_1(N)=0$. Next, recall $H^i(K;\bZ)=0$, $i\geq 1$, and so $0=H^i(K;\bZ)\cong H^i(\overline{U};\bZ)$, $i\geq 1$. We note from the long exact sequence for homology pairs that $H_1(N,\partial N;\bZ)\cong\widetilde{H}_0(\partial N;\bZ) = 0$, where the last equality follows since $\partial N$ is connected. By Alexander duality, $0=H^{n-i-1}(\overline{U})\cong H_i(N)$, for $1\leq i \leq n-2$. Poincar\'e--Lefschetz duality and the universal coefficient theorem give $H_{n-1}(N;\bZ)\cong H^1(N,\partial N;\bZ)\cong \mathrm{Hom}(H_1(N,\partial N;\bZ),\bZ) =0$, where the last equality holds since $H_1(N,\partial N;\bZ)=0$. Finally, as $N$ is a compact $n$-manifold with boundary, we have $H_i(N)=0$, $i\geq n$. In summary, $N$ satisfies $\pi_1(N)=0$ and $H_i(N)=0$, $i\geq 1$ and so by the Hurewicz theorem and the Whitehead theorem $N$ is contractible.
    
    Next, by \cite[Corollary 2.5]{GuilbaultSomeCompactContractible}, $\pi_i( \overline{U})\cong\pi_i(\partial\overline{U}) \cong \pi_i(\partial N)$ for $1\leq i \leq n-4$ and $\pi_i( \overline{U})\cong \pi_i(K)$. By \Cref{p: Higman} and \Cref{boundaryconnectedness}, this completes the existence of the manifold $N$ with the desired topological properties.

    Now by \cite[Lemma 2.4 and Corollary 2.5]{GuilbaultSomeCompactContractible}, $\partial \overline{U}\hookrightarrow \overline{U}$ is 2-connected; therefore, by \cite[Theorem 2.7 and Remark 2.8]{FrenckSpacesofPSC} (see also \cite[Theorem E]{EbertRandallWilliamsInfiniteLoopSpaces}), $N=\bS^n\setminus \mathrm{int}(\overline{U})$ admits a positive scalar curvature metric with product boundary. Attaching an infinite cylindrical end to $N$ produces a complete metric on $\mathrm{int}(N)$ with uniform positive scalar curvature and Ricci curvature bounded below.
\end{proof}

\printbibliography
\end{document}